\documentclass[11pt,a4paper]{article}
\usepackage[top=3cm, bottom=3cm, left=3cm, right=3cm]{geometry}
\usepackage{tabu}
\usepackage{breqn}
\usepackage[hang,flushmargin]{footmisc} %nonindent footnote

\usepackage{amsmath,amsthm,amssymb,graphicx, multicol, array}
\usepackage{enumerate}
\usepackage{enumitem}
\usepackage{hyperref}
\usepackage{setspace}
\usepackage{xcolor}
\usepackage{currfile}

\usepackage{tabto}

\DeclareMathOperator{\li}{li}

\newtheorem{thm}{Theorem}[section]
\newtheorem{lem}{Lemma}[section]

\newtheorem{exe}{Exercise}[section]

\newcommand{\N}{\mathbb{N}}
\newcommand{\Z}{\mathbb{Z}}

\newcommand{\C}{\mathbb{C}}

\usepackage{fancyhdr}
\title{Euler Totient Function And The Largest Integer Function Over The Shifted Primes}
\date{}
\author{N. A. Carella}

\begin{document}
%\doublespacing
\thispagestyle{empty}
\date{}

\maketitle
\textbf{\textit{Abstract}:} Let $ x\geq 1 $ be a large number, let $ [x]=x-\{x\} $ be the largest integer function, and let $ \varphi(n)$ be the Euler totient function. The asymptotic formula for the new finite sum over the primes $ \sum_{p\leq x}\varphi([x/p])=(6/\pi^2)x\log \log x+c_1x+O\left (x(\log x)^{-1}\right) $, where $c_1$ is a constant, is evaluated in this note. 
\let\thefootnote\relax\footnote{\today  
}

\tableofcontents
%sssssssssssssssssssssssssssssssssssssssssssssssssssssssssssssssssssssssssssssssssssssssssssssssssssss
%sssssssssssssssssssssssssssssssssssssssssssssssssssssssssssssssssssssssssssssssssssssssssssssssssssss
\section{Introduction} \label{s0800}
The average order $\sum_{n\leq x}d(n)$ of the divisors counting function $d(n)=\#\{d\mid n\}$ has a direct representation in term of the average order of the largest integer function $ [x]=x-\{x\} $ as
\begin{equation}\label{eq3800.050}
\sum_{n\leq x}d(n)=\sum_{n\leq x}\left [\frac{x}{n}\right ].
\end{equation}
Following this equivalence, the average orders $\sum_{n\leq x}f(n)$ of some arithmetic functions were extended to $\sum_{n\leq x}f([x/n])$ in \cite{BS2018}, \cite{CN2021}, \cite{ZW2020}, et alii. Continuing these equivalences, the average order $\sum_{n\leq x}\omega(n)$ of the prime divisors counting function $\omega(n)=\#\{p\mid n\}$ has a direct representation in term of the average order of largest integer function as
\begin{equation}\label{eq3800.050}
\sum_{n\leq x}\omega(n)=\sum_{p\leq x}\left [\frac{x}{p}\right ].
\end{equation}   
Further, this equivalence introduces a new phenomenon, it changes the indices from the integer domain to the prime domain. The closely related Titchmarsh divisor problem seems to have a complicated representation in term of the average order of largest integer function as
\begin{equation}\label{eq3800.060}
\sum_{p\leq x}d\left (p-1\right )\stackrel{?}{=} \sum_{p\leq x}\left [\frac{x}{p-1}\right ]f(p),
\end{equation}
where $f(n)$ is a function. \\

Let $f,g,h: \N\longrightarrow \C$ be multiplicative functions. Some analytic techniques for evaluating the finite sums $\sum_{p\leq x}f\left ([x/p]\right )$ for multiplicative functions defined by Dirichlet convolutions $f(n)=\sum_{d\mid n}g(d)h(n/d)$, and having fast rates of growth approximately $ f(n)\gg n(\log n)^b $, for some $ b\in \Z $, are introduced here. These elementary methods used within are simpler and about four fold more efficient than the analytic methods used in the current literature as \cite{ZW2020}, et alii, to evaluate the simpler sums $\sum_{n\leq x}f([x/n])$. As a demonstration, the asymptotic formula for the new finite sum $ \sum_{p\leq x}\varphi([x/p])$ is determined in Theorem \ref{thm3800.001}.

%sssssssssssssssssssssssssssssssssssssssssssssssssssssssssssssssssssssssssssssssssssssssssssssssssssss
%sssssssssssssssssssssssssssssssssssssssssssssssssssssssssssssssssssssssssssssssssssssssssssssssssssss  $  $
%sssssssssssssssssssssssssssssssssssssssssssssssssssssssssssssssssssssssssssssssssssssssssssssssssssss
%sssssssssssssssssssssssssssssssssssssssssssssssssssssssssssssssssssssssssssssssssssssssssssssssssssss  
\section{Euler Totient Function Over The Shifted Primes}\label{s3801}
The Euler totient function is defined by $ \varphi(n)=n\sum_{d\mid n}\mu(d)/d $, and other identities. It is multiplicative and satisfies the growth condition $ \varphi(n)\gg n/\log \log n $. The average order over the integer domain has the form
\begin{equation}\label{eq3800.050}
\sum_{n\leq x}\varphi(n)= \frac{3}{\pi^2}x^2+O\left (  x\log x\right ),
\end{equation}
see \cite[Theorem 3.7]{AP1976}, and similar references, and average order over the prime domain has the form
\begin{equation}\label{eq3800.060}
\sum_{p\leq x}\varphi\left (p-1\right )= \li(x^2)\prod_{p\geq 2}\left (1-\frac{1}{p(p-1)}\right )+O\left (\frac{x^2}{(\log x)^B}\right ),
\end{equation}
where $B>1$ is an arbitraryl constant, see \cite{PS1941}. The new finite sum
\begin{equation}\label{eq3800.070}
\sum_{n\leq x}\varphi\left (\left [\frac{x}{n}\right ]\right )= \frac{6}{\pi^2}x\log x+O\left (x\log \log x\right),
\end{equation}
is evaluated in \cite{CN2021}, and a weaker version in \cite{ZW2020}. Standard analytic techniques are used here to assemble the first proof of the asymptotic formula for the new finite sum $\sum_{p\leq x}\varphi([x/p] )  $.

\begin{thm}\label{thm3800.001} If $ x\geq 1 $ is a large number, then, 
\begin{equation}\label{eq3800.001}
\sum_{p\leq x}\varphi\left (\left [\frac{x}{p}\right ]\right )= c_0x\log \log x+c_1x+O\left (\li(x)\right),
\end{equation}
where $c_0=6/\pi^2$, and $c_1>0$ are constants.
\end{thm}
\begin{proof} Use the identity $ \varphi(n)=n\sum_{d\mid n}\mu(d)/d $ to rewrite the finite sum, and switch the order of summation:
\begin{eqnarray}\label{eq3800.010}
\sum_{p\leq x}\varphi\left (\left [\frac{x}{p}\right ]\right )
&=& \sum_{p\leq x} \left [\frac{x}{p}\right]\sum_{d\,\mid\, [x/p]}\frac{\mu(d)}{d} \\
&=& \sum_{d\leq x} \frac{\mu(d)}{d}\sum_{\substack{p\leq x\\d\,\mid\, [x/p]}}\left [\frac{x}{p}\right] \nonumber.
\end{eqnarray}

Apply Lemma \ref{lem3802.050} to remove the congruence on the inner sum index, and break it up into two subsums. Specifically,
\begin{eqnarray}\label{eq3800.015}
\sum_{p\leq x} \frac{\mu(d)}{d}\sum_{\substack{p\leq x\\d\,\mid\, [x/p]}}\left [\frac{x}{p}\right]
&=&\sum_{d\leq x} \frac{\mu(d)}{d}\sum_{p\leq x}\left [\frac{x}{p}\right]\cdot \frac{1}{d}\sum_{0\leq s\leq d-1}e^{i2\pi s[x/p]/d} \\
&=&\sum_{d\leq x} \frac{\mu(d)}{d^2}\sum_{p\leq x}\left [\frac{x}{p}\right]\nonumber \\
&&\hskip 0.250 in +\sum_{d\leq x} \frac{\mu(d)}{d^2}\sum_{p\leq x}\left [\frac{x}{p}\right]\sum_{0< s\leq d-1}e^{i2\pi s[x/p]/d}\nonumber \\
&=&M(x) + E(x)\nonumber.
\end{eqnarray}
The main term $M(x)$ is computed in Lemma \ref{lem3802.100} and the error term $E(x)$ is computed in Lemma \ref{lem3802.200}. Summing these expressions yields
\begin{eqnarray}\label{eq3800.020}
\sum_{p\leq x}\varphi\left (\left [\frac{x}{p}\right ]\right )&=&M(x) + E(x)\\
&=&c_0x\log \log x+c_1x+c_2\li(x) +O\left (xe^{-c\sqrt{\log x}}\right ) +O\left (\li(x)\right )
\nonumber\\
&=&c_0x\log \log x+c_1x+O\left (\li(x)\right )
\nonumber,
\end{eqnarray}
where $c_0=6/\pi^2, c_1,c_2$, and $c>0$ are constants.
\end{proof}

The finite sum $ \sum_{p\leq x}\varphi([x/p])$ is the canonical representative of this class of finite sums associated with the totient function. The leading term in the asymptotic formula for the more general sum $ \sum_{p\leq x}\varphi([x/(p+a)])$, where $a\ne-p$ is an integer parameter, is independent of the small parameter $a\in \Z$. However, the parameter $a$ does contribute to the secondary terms. Other classes of arithmetic functions having the same leading terms in the asymptotic formulas of the average orders are described in \cite[Section 3.1]{OT2021}.\\

Under the RH, the optimal evaluation is expected to be of the form
\begin{equation}\label{eq3800.025}
\sum_{p\leq x}\varphi\left (\left [\frac{x}{p+a}\right ]\right )=c_0 x\log \log x+c_1x+c_2\li(x) +O\left (x^{1/2}(\log x)^2\right ),
\end{equation}
where $c_0=6/\pi^2$, $c_1=c_1(a) $ and $c_2=c_2(a)$ are constants, depending on the parameter $a$.

%sssssssssssssssssssssssssssssssssssssssssssssssssssssssssssssssssssssssssssssssssssssssssssssssssssss
%sssssssssssssssssssssssssssssssssssssssssssssssssssssssssssssssssssssssssssssssssssssssssssssssssssss  
\section{Foundation Results for the Phi Function}\label{s3802}
The detailed and elementary proofs of the preliminary results required in the proof of Theorem \ref{thm3800.001} concerning the Euler phi function $\varphi(n)=\sum_{d\mid n}\mu(d)d$ are recorded in this section.\\

\subsection{Preliminary Results}
\begin{lem}\label{lem3802.050}  Let $ x\geq 1 $ be a large number, and let $1\leq d,m, n\leq x$ be integers. Then,
\begin{equation}\label{eq3802.010}
\frac{1}{d}\sum_{0\leq s\leq d-1}e^{i2\pi ms/d} =\left \{\begin{array}{ll}
1 & \text{ if } d\mid m,  \\
0& \text{ if } d\nmid m, \\
\end{array} \right .
\end{equation}
\end{lem}

\begin{lem}\label{lem3802.060} Let $m\leq x$ be a fixed integer, and let $ x\geq 1 $ be a large number. Then,
\begin{enumerate}%[font=\normalfont, label=(\roman*)]
\item $\displaystyle  \sum_{\substack{n\leq x\\n\mid m}} \frac{\mu(n)}{n}=O\left (1 \right )$.
\item $\displaystyle \sum_{\substack{n\leq x\\n\mid m}} \frac{\mu(n)}{n^2}=A_0(m)+O\left (\frac{1 }{ x}\right ),$
\item $\displaystyle \sum_{\substack{n\leq x\\n\nmid m}} \frac{\mu(n)}{n^2}=A_1(m)+O\left (\frac{1 }{ x}\right ),$
\end{enumerate}
where $|A_0(m)|<2$, and $|A_1(m)|<2$ are constants depending on $m\geq1$.
\end{lem}
\begin{proof} (i) Since $m\leq x$, the upper bound of the first finite sum is as follows.
\begin{equation}\label{eq3802.070}
\sum_{\substack{n\leq x\\n\mid m}}\frac{\mu(n)}{n}=\sum_{n\mid m}\frac{\mu(n)}{n}=\prod_{\substack{p\leq x\\r\mid m}}\left (1-\frac{1}{r}\right )=O\left ( 1\right ),
\end{equation}
where $r\geq 2$ is prime. The other relations (ii) and (iii) are routine calculations.
\end{proof}

\subsection{The Main Term $S(x)$}
\begin{lem}\label{lem3802.100} If $ x\geq 1 $ is a large number, then,
\begin{equation}\label{eq3802.110}
\sum_{d\leq x} \frac{\mu(d)}{d^2}\sum_{p\leq x}\left [\frac{x}{p}\right]
= \frac{6}{\pi^2}x\log \log x+c_1x+c_2\li(x) +O\left (xe^{-c\sqrt{\log x}}\right ),
\end{equation}
where $c_1, c_2$, and $c>0$ are constants.
\end{lem}
\begin{proof}Expand the bracket and evaluate the two subsums. Specifically, the first subsum is
\begin{eqnarray}\label{eq3802.115}
x\sum_{d\leq x} \frac{\mu(d)}{d^2}\sum_{p\leq x}\frac{1}{p}
&=&x\left (\frac{6}{\pi^2} +O\left (\frac{1}{x}\right )\right )\left (\log \log x+B_1+O\left (\frac{1}{x}\right )\right )\\&=&\frac{6}{\pi^2}x\log \log x+c_1x+O\left (\log \log x\right )\nonumber,
\end{eqnarray}
and the second subsum is 
\begin{eqnarray}\label{eq3802.120}
-\sum_{d\leq x} \frac{\mu(d)}{d^2}\sum_{p\leq x}\left \{\frac{x}{p}\right\}
&=&-\left (\frac{6}{\pi^2} +O\left (\frac{1}{x}\right )\right )\left ((1-\gamma)\li(x) +O\left (xe^{-c\sqrt{\log x}}\right )\right )\\
&=&c_2\li(x) +O\left (xe^{-c\sqrt{\log x}}\right )\nonumber,
\end{eqnarray}
where $\li(x)$ is the logarithm integral, $c_1, c_2$, and $ B_1$ are constants, and $c>0$ is an absolute constant. Summing \eqref{eq3802.115} and \eqref{eq3802.120} completes the verification.
\end{proof}
The optimal error term in \eqref{eq3802.110} is the same as the optimal error term in the prime number theorem, see \cite[Corollary 1]{SL1976}. This is a direct consequence of the optimal evaluation of the basic sum of fractional parts
\begin{equation}\label{eq3802.044}
\sum_{p\leq x} \left \{\frac{x}{p}\right\}=(1-\gamma)\li(x) +O\left (xe^{-c\sqrt{\log x}}\right ), 
\end{equation}
where $ \gamma $ is Euler constant, see \cite[Theorem 0]{PF2010}.

\subsection{The Error Term $E(x)$}
\begin{lem}\label{lem3802.200} Let $ x\geq 1 $ be a large number. Then,
\begin{equation}\label{eq3802.210}
\sum_{d\leq x} \frac{\mu(d)}{d^2}\sum_{p\leq x}\left [\frac{x}{p}\right]\sum_{0< s\leq d-1}e^{i2\pi s[x/p]/d}
= ax+O\left (\frac{x}{\log x}\right),
\end{equation}
where $a\ne0$ is a constant.
\end{lem}
\begin{proof} Let $\pi(x)=\#\{\text{prime } p\leq x\}$ be the primes counting function, let $\li(x)$ be the logarithm integral, and let $p_k$ be the $k$th prime in increasing order. The sequence of values
\begin{equation}\label{eq3802.225}
\left [\frac{x}{p_{k}}\right]=\left [\frac{x}{p_{k+1}}\right]=\cdots =\left [\frac{x}{p_{k+r}}\right]
\end{equation}
arises from the sequence of primes $x/(n+1)\leq p_{k}, p_{k+1}, \ldots, p_{k+r}\leq x/n$. Therefore, the value $m=[x/p]\geq1$ is repeated 
\begin{equation}\label{eq3802.220}
\pi\left (\left [\frac{x}{n}\right]\right)-\pi\left (\left [\frac{x}{n+1}\right]\right)=\frac{\li(x)}{n(n+1)} +O\left (\frac{x}{n}e^{-c\sqrt{\log x}}\right ) 
\end{equation}
times as $p$ ranges over the prime values in the interval $[x/(n+1), x/n]$. Hence, substituting \eqref{eq3802.220} into the triple sum $E(x)$, and reordering it yield
\begin{eqnarray}\label{eqeq3802.230}
E(x)&=&\sum_{p\leq x}\left [\frac{x}{p}\right]\sum_{d\leq x} \frac{\mu(d)}{d^2}\sum_{0< a\leq d-1}e^{i2\pi m/d}\\
&=&\sum_{p\leq x}\left ( \frac{\li(x)}{n(n+1)}+O\left (\frac{x}{n}e^{-c\sqrt{\log x}}\right )\right)\sum_{d\leq x} \frac{\mu(d)}{d^2}\sum_{0< a\leq d-1}e^{i2\pi am/d}\nonumber\\
&=&\li(x)\sum_{n\leq x}\frac{1}{n(n+1)}\sum_{d\leq x} \frac{\mu(d)}{d^2}\sum_{0< a\leq d-1}e^{i2\pi am/d}\nonumber\\
&& \hskip 1.5 in + O\left (xe^{-c\sqrt{\log x}}\sum_{n\leq x} \frac{1}{n}\sum_{d\leq x} \frac{\mu(d)}{d^2}\sum_{0< a\leq d-1}e^{i2\pi am/d}\right)\nonumber\\ 
&=&E_{0}(x)+E_{1}(x)\nonumber.
\end{eqnarray}
The finite subsums $E_0(x)$, estimated in Lemma \ref{lem3802.500}, and $E_1(x)$, estimated in Lemma \ref{lem3802.400}, correspond to the subsets of integers $p\leq x$ such that $d\mid [x/p]$, and $d\nmid [x/p]$, respectively. Summing yields 
\begin{eqnarray}\label{eq3802.240}
E(x)&=&E_{0}(x)+E_{1}(x)\\
&=& O\left (\li(x)\right )+O\left (xe^{-c\sqrt{\log x}}\right )\nonumber\\
&=&O\left (\li(x)\right )\nonumber,
\end{eqnarray}
where $c>0$ is an absolute constant.
\end{proof}

\subsection{The Sum $E_{0}(x)$}
\begin{lem}\label{lem3802.500} Let $ x\geq 1 $ be a large number, let $[x]=x-\{x\}$ be the largest integer function, and $m=[x/p]\leq [x/n]\leq x$. Then,
\begin{equation}\label{eq3802.510}
\li( x)\sum_{n\leq x}\frac{1}{n(n+1)}\sum_{d\leq x} \frac{\mu(d)}{d^2}\sum_{0< a\leq d-1}e^{i2\pi am/d}
= O(\li( x)).
\end{equation}
\end{lem}
\begin{proof} The set of values $m=[x/p]\leq [x/n]\leq x$ such that $d\mid m$. Evaluating the incomplete function returns
\begin{eqnarray}\label{eq3802.520}
E_{0}(x)&=&\li( x)\sum_{n\leq x}\frac{1}{n(n+1)}\sum_{d\leq x} \frac{\mu(d)}{d^2}\sum_{0< a\leq d-1}e^{i2\pi am/d}\\
&=& \li( x)\sum_{n\leq x}\frac{1}{n(n+1)}\sum_{\substack{d\leq x\\d\mid m}} \frac{\mu(d)}{d^2}\cdot (d-1)\nonumber\\
&=&\li( x)\sum_{n\leq x}\frac{1}{n(n+1)}\sum_{\substack{d\leq x\\d\mid m}} \frac{\mu(d)}{d}-\li( x)\sum_{n\leq x}\frac{1}{n(n+1)}\sum_{\substack{d\leq x\\d\mid m}} \frac{\mu(d)}{d^2}\nonumber\\
&=&E_{00}(x)+E_{01}(x)\nonumber.
\end{eqnarray}
The first term has the upper bound
\begin{eqnarray}\label{eq3802.530}
E_{00}(x)&=&\li( x)\sum_{n\leq x}\frac{1}{n(n+1)}\sum_{\substack{d\leq x\\d\mid m}} \frac{\mu(d)}{d}\\
&=&\li( x)\sum_{n\leq x}\frac{1}{n(n+1)}\sum_{d\mid m} \frac{\mu(d)}{d}\nonumber\\
& \ll &\li( x)\nonumber.
\end{eqnarray}
This follows from the trivial inequality
\begin{equation}\label{eq3802.540}
\sum_{\substack{d\leq x\\d\mid m}} \frac{\mu(d)}{d}=\sum_{d\mid m} \frac{\mu(d)}{d}= \frac{\varphi(m)}{m}\leq 1,
\end{equation}
where $m=[x/p] \leq [x/n]\leq x$. The second term has the upper bound
\begin{eqnarray}\label{eq3802.560}
E_{01}(x)&=&\li( x)\sum_{n\leq x}\frac{1}{n(n+1)}\sum_{\substack{d\leq x\\d\mid m}} \frac{\mu(d)}{d^2}\\
&=&\li( x)\sum_{n\leq x}\frac{1}{n(n+1)}\sum_{d\mid [x/n]} \frac{\mu(d)}{d^2}\nonumber\\
&\ll & \li( x)\nonumber.
\end{eqnarray}
Summing yields $E_{0}(x)=E_{00}(x)+E_{01}(x)=O(\li( x))$.
\end{proof}

\subsection{The Sum $E_{1}(x)$}
\begin{lem}\label{lem3802.400} If $ x\geq 1 $ is a large number, then,
\begin{equation}\label{eq3802.410}
xe^{-c_0\sqrt{\log x}}\sum_{n\leq x} \frac{1}{n}\sum_{d\leq x} \frac{\mu(d)}{d^2}\sum_{0< a\leq d-1}e^{i2\pi am/d}
= O\left (xe^{-c\sqrt{\log x}}\right )\nonumber,
\end{equation}
where $c_0>0$ and $c>0$ are absolute constants.
\end{lem}

\begin{proof} The absolute value provides an upper bound:
\begin{eqnarray}\label{eq3802.420}
\left |E_{1}(x)\right |&=&\left | xe^{-c_0\sqrt{\log x}}\sum_{n\leq x} \frac{1}{n}\sum_{d\leq x} \frac{\mu(d)}{d^2}\sum_{0< a\leq d-1}e^{i2\pi am/d}\right |\\
&\leq & xe^{-c_0\sqrt{\log x}}\sum_{n\leq x} \frac{1}{n}\sum_{d\leq x} \frac{1}{d}\nonumber\\
&=&O\left ((x\log ^2 x)e^{-c_0\sqrt{\log x}}\right )\nonumber\\
&=&O\left (xe^{-c\sqrt{\log x}}\right )\nonumber,
\end{eqnarray}
where $c_0>0$ and $c>0$ are absolute constants.
\end{proof}

%sssssssssssssssssssssssssssssssssssssssssssssssssssssssssssssssssssssssssssssssssssssssssssssssssssss
%sssssssssssssssssssssssssssssssssssssssssssssssssssssssssssssssssssssssssssssssssssssssssssssssssssss
%sssssssssssssssssssssssssssssssssssssssssssssssssssssssssssssssssssssssssssssssssssssssssssssssssssss
%sssssssssssssssssssssssssssssssssssssssssssssssssssssssssssssssssssssssssssssssssssssssssssssssssssss
\section{Numerical Data}\label{S3505}
Small numerical tables were generated by an online computer algebra system, the range of numbers $ x\leq 10^6
 $ is limited by the wi-fi bandwidth. For each parameter $a$, the error term for the finite sum over the shifted primes is defined by
\begin{equation}\label{eq3505.033}
E(x,a)=\sum_{p\leq x}\varphi\left (\left [\frac{x}{p+a}\right ]\right )- \frac{6}{\pi^2}x\log \log x.
\end{equation}
The prime values $a=-p$ are not used since \eqref{eq3505.033} has a singularity at this point. The tables are for the parameters $a=0$, $a=4$, and $a=-4$ respectively. All the calculations are within the predicted ranges $ E(x,a)=O(x) $.
\begin{table}[h!]
\centering
\caption{Numerical Data For $\sum_{p\leq x}\varphi([x/p])$.} \label{t3505.003}
\begin{tabular}{l|l|r| r}
$x$&$\sum_{p\leq x}\varphi([x/p])$&$6\pi^{-2}x\log \log x$&Error $E(x,0)$\\
\hline
10&$8$&   $5.07$   &$-2.93$\\
100&$94$&  $92.84$   &$1.96$\\
1000&$1115$&   $1174.91$   &$-59.91$\\
10000&$12891$&  $13497.97$   &$-606.97$\\
100000 & $147771$ &   $148545.18$   &$-774.20$\\
1000000&$1526405$&   $1596290.10$   &$-69885.10$\\
\end{tabular}
\end{table}

\begin{table}[h!]
\centering
\caption{Numerical Data For $\sum_{p\leq x}\varphi([x/(p-4)])$.} \label{t3505.009}
\begin{tabular}{l|l|r| r}
$x$&$\sum_{p\leq x}\varphi([x/(p-4)])$&$6\pi^{-2}x\log \log x$&Error $E(x,-4)$\\
\hline
10&$14$&   $5.07$   &$8.93$\\
100&$167$&  $92.84$   &$74.16$\\
1000&$1868$&   $1174.91$   &$693.09$\\
10000&$20537$&  $13497.97$   &$7039.03$\\
100000&$224901$&   $148545.18$   &$76355.81$\\
1000000&$2244876$&   $1596290.10$   &$648585.93$\\
\end{tabular}
\end{table}

\begin{table}[h!]
\centering
\caption{Numerical Data For $\sum_{p\leq x}\varphi([x/(p+4)])$.} \label{t3505.009}
\begin{tabular}{l|l|r| r}
$x$&$\sum_{p\leq x}\varphi([x/(p+4)])$&$6\pi^{-2}x\log \log x$&Error $E(x,4)$\\
\hline
10&$3$&   $5.07$   &$-2.07$\\
100&$58$&  $92.84$   &$-34.84$\\
1000&$791$&   $1174.91$   &$-383.91$\\
10000&$8956$&  $13497.97$   &$-4541.97$\\
100000&$113334$&   $148545.18$   &$-35211.19$\\
1000000&$1225300$&   $1596290.10$   &$-370990.07$\\
\end{tabular}
\end{table}

%PPPPPPPPPPPPPPPPPPPPPPPPPPPPPPPPPPPPPPPPPPPPPPPPPPPPPPPPPPPPPPPPPPPPPPPPPPPPPPPPPPPPPPPPPPPPPPPPPPPPPPPPPPPPPPPPPPPPPPPP
%PPPPPPPPPPPPPPPPPPPPPPPPPPPPPPPPPPPPPPPPPPPPPPPPPPPPPPPPPPPPPPPPPPPPPPPPPPPPPPPPPPPPPPPPPPPPPPPPPPPPPPPPPPPPPPPPPPPPPPPP
%PPPPPPPPPPPPPPPPPPPPPPPPPPPPPPPPPPPPPPPPPPPPPPPPPPPPPPPPPPPPPPPPPPPPPPPPPPPPPPPPPPPPPPPPPPPPPPPPPPPPPPPPPPPPPPPPPPPPPPPP
%PPPPPPPPPPPPPPPPPPPPPPPPPPPPPPPPPPPPPPPPPPPPPPPPPPPPPPPPPPPPPPPPPPPPPPPPPPPPPPPPPPPPPPPPPPPPPPPPPPPPPPPPPPPPPPPPPPPPPPPP
\section{Open Problems}\label{exe3505}

\begin{exe}\label{exe3510} {\normalfont Let $\mu: \N\longrightarrow \{-1,0,1\}$ be the Mobius function, and let $q\geq 1$ be a fixed integer. Prove or disprove the following asymptotic formula.
$$\sum_{\substack{n\leq x\\n\mid q}}\mu(n)=O\left ( xe^{-c\sqrt{\log x}}\right ),
$$
where $c>0$ is an absolute constant. A weaker estimate is cited in Lemma \ref{lem3802.200}.}
\end{exe}

\begin{exe}\label{exe3520} {\normalfont Let $\lambda: \N\longrightarrow \{-1,1\}$ be the Liouville function, and let $q\geq 1$ be a fixed integer. Prove or disprove the following asymptotic formula.
$$\sum_{\substack{n\leq x\\n\mid q}}\lambda(n)=O\left ( xe^{-c/\sqrt{\log x}}\right ),
$$
where $c>0$ is a constant.}
\end{exe}

\begin{exe}\label{exe3530} {\normalfont Let $\sigma: \N\longrightarrow \N$ be the sum of divisors function, and let $q\geq 1$ be a fixed integer. Prove the following explicit upper bound.
$$\sum_{\substack{n\leq x\\n\mid q}}\sigma(n)\leq 2q\log \log q.
$$
}
\end{exe}

\begin{exe}\label{exe3540} {\normalfont Let $\varphi: \N\longrightarrow \N$ be the totient function, and let $x\geq 1$ be a large number. Determine an asymptotic formula for the following finite sum.
$$\sum_{p\leq x} \frac{[x/p]}{\varphi\left ([x/p]\right )}.
$$
}
\end{exe}

\begin{exe}\label{exe3540} {\normalfont Let $\lambda: \N\longrightarrow \N$ be the Carmichel function, and let $x\geq 1$ be a large number. Determine an asymptotic formula for the following finite sum.
$$\sum_{p\leq x} \lambda\left ( [x/p]\right ).$$
}
\end{exe}

\begin{exe}\label{exe3570} {\normalfont Let $\lambda: \N\longrightarrow \N$ be the Carmichel function, let $\{x\}$ be the fractional function, and let $x\geq 1$ be a large number. Verify the asymptotic formula for the following finite sum.
$$\sum_{n\leq x} \left \{\frac{\lambda(n+1)}{n}\right \}=\frac{x}{\log x}e^{b\log \log x/\log \log \log x)}-\li(x)+O\left (e^{-c\sqrt{\log x}}\right),$$
where $\li(x)$ is the logarithm integral, and $b>0$, and $c>0$ are constants. Hint: $\{x\}=x-[x]$.
}
\end{exe}

\begin{exe}\label{exe3575} {\normalfont Let $\varphi: \N\longrightarrow \N$ be the totient function, let $\{x\}$ be the fractional function, and let $x\geq 1$ be a large number. Verify the asymptotic formula for the following finite sum.
$$\sum_{n\leq x} \left \{\frac{\varphi(n+1)}{n}\right \}=\frac{3}{\pi^2}x-\li(x)+O\left (e^{-c\sqrt{\log x}}\right),$$
where $\li(x)$ is the logarithm integral, and $c>0$ is a constant.
}
\end{exe}

\begin{exe}\label{exe3575} {\normalfont Let $\sigma: \N\longrightarrow \N$ be the sum of divisors function, let $\{x\}$ be the fractional function, and let $x\geq 1$ be a large number. Verify the asymptotic formula for the following finite sum.
$$\sum_{n\leq x} \left \{\frac{n}{\sigma(n)-1}\right \}=ax-\li(x)+O\left (e^{-c\sqrt{\log x}}\right),$$
where $\li(x)$ is the logarithm integral, and $a>0$ and $c>0$ are constants.
}
\end{exe}
%BBBBBBBBBBBBBBBBBBBBBBBBBBBBBBBBBBBBBBBBBBBBBBBBBBBBBBBBBBBBBBBBBBBBBBBBBBBBBBBBBBBBBBBBBBBBBBBBBBBBBBBBBBBBBBBBBBBBBBB%BBBBBBBBBBBBBBBBBBBBBBBBBBBBBBBBBBBBBBBBBBBBBBBBBBBBBBBBBBBBBBBBBBBBBBBBBBBBBBBBBBBBBBBBBBBBBBBBBBBBBBBBBBBBBBBBBBBBBBB
%BBBBBBBBBBBBBBBBBBBBBBBBBBBBBBBBBBBBBBBBBBBBBBBBBBBBBBBBBBBBBBBBBBBBBBBBBBBBBBBBBBBBBBBBBBBBBBBBBBBBBBBBBBBBBBBBBBBBBBB

\currfilename.\\

\end{document}